\documentclass[a4paper]{article}
\usepackage{amsmath,amsthm,amssymb,mathtools}
\usepackage{enumitem}
\usepackage{changepage}
\usepackage{xcolor}
\usepackage{thmtools, thm-restate}
\usepackage{cleveref}

\newcommand{\suchthat}{\mid}

\newcommand{\from}{\colon}
\newcommand{\boundary}{\partial}
\newcommand{\union}{\cup}
\newcommand{\intersection}{\cap}
\newcommand{\bigunion}{\bigcup}
\newcommand{\bigintersection}{\bigcap}
\newcommand{\composed}{\circ}

\newcommand{\naturals}{\mathbb{N}}

\newcommand{\gromov}[3]{\ensuremath{\left(#2\cdot#3\right)_{#1}}}

\newcommand{\set}[1]{\left\{#1\right\}}
\newcommand{\restricted}[1]{\vert_{#1}}
\newcommand{\abs}[1]{\left|#1\right|}

\DeclarePairedDelimiter\floor{\lfloor}{\rfloor}
\newcommand{\isom}{\cong}
\newcommand{\disjointunion}{\amalg}

\DeclareMathOperator{\Diam}{Diam}
\DeclareMathOperator{\image}{Image}
\DeclareMathOperator{\Skel}{Sk}
\DeclareMathOperator{\sd}{sd}
\DeclareMathOperator{\dist}{d}
\DeclareMathOperator{\newdist}{\widehat{d}}
\DeclareMathOperator{\vdistance}{\rho}

\newcounter{dummy}\numberwithin{dummy}{section}

\theoremstyle{plain}
\newtheorem{lemma}[dummy]{Lemma}
\newtheorem{corollary}[dummy]{Corollary}
\newtheorem{theorem}[dummy]{Theorem}
\newtheorem{proposition}[dummy]{Proposition}

\newtheorem{question}[dummy]{Question}

\theoremstyle{remark}

\newtheorem{remark}[dummy]{Remark}

\theoremstyle{definition}
\newtheorem{definition}[dummy]{Definition}

\title{Local simple connectedness of boundaries of hyperbolic groups}
\author{Benjamin Barrett}

\begin{document}
\maketitle

\begin{abstract}
  In this paper we prove a theorem describing the local topology of the
  boundary of a hyperbolic group in terms of its global topology: the boundary
  is locally simply connected if and only if the complement of any point in the
  boundary is simply connected. This generalises a theorem of Bestvina and
  Mess~\cite{bestvinamess91}.
\end{abstract}

\section{Introduction}

If a group acts geometrically on a proper hyperbolic metric space then,
following Gromov~\cite{gromov87}, we say that the group is hyperbolic, and
quasi-isometry invariants of the space can be thought of as invariants of the
group. One particularly powerful quasi-isometry invariant that can be used in
this way is the Gromov boundary of the space. This a topological space defined
as the set of equivalence classes of geodesic rays in the space emanating from
a fixed base point, where rays are defined to be equivalent if they remain a
bounded distance apart. In this way one can define the boundary $\boundary G$
of a hyperbolic group $G$.

The boundary of a hyperbolic group is useful as a quasi-isometry invariant of
the group: if groups $G_1$ and $G_2$ are quasi-isometric then $\boundary G_1
\isom \boundary G_2$. Furthermore, certain topological properties of the
boundary translate directly into algebraic properties of the group. By a
theorem of Stallings~\cite{stallings68}, $\boundary G$ is disconnected if and
only if $G$ admits a splitting as an amalgamated product or HNN extension over
a finite subgroup. A theorem of Bowditch further characterises splittings over
virtually cyclic subgroups: if $\boundary G$ is connected then $G$ splits over
a virtually cyclic subgroup if and only if $\boundary G$ can be separated by
the removal of a pair of points. Sometimes the topology of $\boundary G$
determines $G$ even more precisely: if $\boundary G$ is homeomorphic to a
circle then the Convergence Group Theorem of Tukia, Casson, Jungreis and
Gabai~\cite{tukia88,cassonjungreis94,gabai92} tells us that $G$ contains a
surface group at finite index.

In~\cite{gromov87}, Gromov showed that the boundary of a hyperbolic group is a
compact metric space with finite topological dimension. Beyond these basic
properties, and excepting the case when $G$ is the fundamental group of a
negatively curved smooth closed manifold, in which case $\boundary G$ is a
sphere, boundaries of hyperbolic groups tend to be wild, fractal-like objects.
Concrete examples include the Menger curve~\cite{benakli92}, which is the
generic case~\cite{champetier95}, Sierpinski compacta, which arise as
boundaries of fundamental groups of compact hyperbolic manifolds with non-empty
totally geodesic boundary, the universal Menger compacta of dimensions
2~\cite{dranishnikov99} and 3~\cite{dymaraosajda07}, and Pontryagin
surfaces~\cite{dranishnikov99}. See also~\cite{davisjanuszkiewicz91} for a
flexible method for constructing interesting high-dimensional boundaries.

In~\cite{kapovichkleiner00}, Kapovich and Kleiner ask the following question.

\begin{question}\cite[Question A]{kapovichkleiner00}
  Which topological spaces are boundaries of hyperbolic groups?
\end{question}

On the one hand, finding concrete examples of spaces arising as boundaries has
proved difficult, but on the other, relatively few restrictions on the topology
of the boundary are known. One restriction of particular importance is a
theorem of Bestvina and Mess~\cite{bestvinamess91}, which states that if
$\boundary G$ is connected and cannot be disconnected by the removal of any
single point then $\boundary G$ is locally path connected. Combined with a
result of Bowditch~\cite{bowditch98b} and Swarup~\cite{swarup96}, this implies
that $\boundary G$ is locally path connected whenever it is connected.

The goal of this paper is to prove the following theorem, which restricts the
local complexity of the boundary of a hyperbolic group in terms of its global
complexity in a similar way to the theorem of Bestvina and Mess described
above.

\begin{theorem}\label{thm:main_theorem}
  Let $G$ be a one-ended hyperbolic group. Then $\boundary G$ is locally simply
  connected if and only if, for every point $\xi \in \boundary G$,  $\boundary
  G - \set\xi$ is simply connected.
\end{theorem}

The paper is structured as follows. In Section~\ref{sec:spheres} we use the
Rips complex of the group $G$ to build a sequence $(K_n)_{n \geq 0}$ of
simplicial complexes that give successively better approximations to the
boundary $\boundary G$. These simplicial complexes are spheres in the Rips
complex. There exist natural simplicial maps $p^n_m\from K_n\to K_m$ for $n \gg
m$, and there exist natural continuous maps $p^\infty_n\from\boundary G \to
K_n$ for all $n$.  The remainder of Section~\ref{sec:spheres} will be spent
proving quantitative statements about these maps.

In Section~\ref{sec:horoball} we prove a compactness result: given a sequence
of vertices $v_n \in K_n$, there exist group elements $g_n$ and a subcomplex
$H$ of the Rips complex such that translation by $g_n$ maps a large
neighbourhood in $K_n$ of $v_n$ onto a large neighbourhood in $H$ of $g_n\cdot
v_n$.  We may further assume that $(g_n)$ converges to a point $\xi \in
\boundary G$, and then we think of $H$ as a horoball in $G$ asymptotic to
$\xi$. We then spend the remainder of Section~\ref{sec:horoball} describing
projection maps $\boundary G - \set\xi \to H$. This relates the topology of
$\boundary G - \set\xi$ to that of $H$, and hence to the topology of large
balls in $K_n$ for $n$ sufficiently large.

In Section~\ref{sec:double_dagger} we begin by recalling the so-called
double-dagger condition $\ddag_M$ of Bestvina and Mess. This condition is
guaranteed to hold if $\boundary G$ does not contain a cut point, and implies
that $\boundary G$ is locally path connected. We then modify the condition to
allow for easier generalisation, and introduce a new condition $\S_M$, which we
will see to be related the local simple connectedness of $\boundary G$. Using
the compactness results from Section~\ref{sec:horoball} we show that this
condition is satisfied if $\boundary G - \set\xi$ is simply connected for any
$\xi$ in $\boundary G$.

Then, in Section~\ref{sec:local_simple_connectedness} we show that our new
condition $\S_M$ implies the local simple connectedness of $\boundary G$. We
begin by using $\S_M$ to construct a sequence of maps $(i^n_\infty)$ from the
2-skeleton of $K_n$ to $\boundary G$. These maps can be thought of as
approximate sections of $p^\infty_n$, in the following sense. For each $n$, let
$r_n\from K_n \to \Skel_2 K_n$ be a (not necessarily continuous) map sending
points to nearby points in the 2-skeleton. Then $i^n_\infty\composed r_n
\composed p^\infty_n$ converges uniformly to the identity map $\boundary G \to
\boundary G$ as $n$ tends to $\infty$. By controlling the diameter of the image
under $i^n_\infty$ of a simplex in $K_n$ we are able to prove that $\boundary
G$ is locally simply connected.

Finally, in Section~\ref{sec:conclusion} we put the results of the previous
sections together to prove Theorem~\ref{thm:main_theorem}, and conclude the
paper with some questions about strengthenings and generalisations.

\subsection*{Acknowledgements}

I am grateful to Henry Wilton for pointing out the proof of
Proposition~\ref{prop:converse_theorem}, and for other helpful comments.  I am
also grateful to John Mackay and Panos Papazoglou for helpful discussions on
this topic.

\section{The Rips complex and the Gromov boundary}\label{sec:spheres}

\subsection{Hyperbolic groups and the Rips complex}

Let $G$ be a hyperbolic group and let $S$ be a symmetric generating set for
$G$. Let $e$ be the identity element of $G$, which we use as a base point. Let
$\dist(\cdot, \cdot)$ be the length metric on the Cayley graph of $G$ with
respect to $S$ obtained by setting all edges to have length $1$; on $G$ this
metric agrees with the word metric. 

We call an isometric embedding of an interval into the Cayley graph a
\emph{geodesic}. If the interval is $(-\infty, \infty)$ we call the geodesic a
\emph{bi-infinite geodesic}; if the interval is $[0, \infty)$ we call it a
\emph{geodesic ray.} A bi-infinite geodesic has two ideal end points in
$\boundary G$, while a geodesic ray has one end point in $G$ and one ideal end
point in $\boundary G$.  Given points $x$ and $y$ in $G \union \boundary G$ we
will often denote by $[x,y]$ a geodesic from $x$ to $y$, thought of as a subset
of the Cayley graph together with its boundary. Geodesics in hyperbolic spaces
are not necessarily unique, so this notation will stand for a choice of
geodesic, but the choice will not matter.

Let $\delta \geq 0$ be large enough to ensure that the Cayley graph of $G$ with
respect to $S$ is $\delta$-hyperbolic, meaning that all geodesic triangles in
the Cayley graph of $G$ are $\delta$-slim. (This means that each edge of each
triangle is contained in the $\delta$-neighbourhood of the union of the other
two edges.) Let $\boundary G$ be the boundary of the group, which can be
identified with the Gromov boundary of the Cayley graph of $G$. To reduce the
growth of constants later on we in fact assume that all \emph{ideal} triangles
(with vertices in $G \union \boundary G$) are $\delta$-hyperbolic; this can be
guaranteed by multiplying $\delta$ by $4$.  

Throughout this and subsequent sections we use~\cite{bridsonhaefliger99} as a
reference for fundamental properties of hyperbolic spaces. The reader is
referred to~\cite{benaklikapovich00} for a useful survey on boundaries of
hyperbolic groups.

\begin{definition}
  Let $D = 10^6\delta + 10^6$.  We denote by $K$ the \emph{Rips complex} of
  $G$; this is the simplicial complex with vertex set $G$ such that a finite
  set of elements of $G$ spans a simplex if and only if that set has diameter
  at most $D$ with respect to the word metric $d$.
\end{definition}

We will sometimes want to refer to distances between points in $K$. The
precise means by which we extend the metric on $G$ to $K$ will not matter
very much, so we make the following simple definition.

\begin{definition}
  For $x$ and $y$ in $K$, choose vertices $x^\star$ and $y^\star$ of the
  minimal simplices containing $x$ and $y$ respectively. Then define
  $\newdist(x,y) = \dist(x^\star, y^\star)$.  
\end{definition}

\begin{remark}\label{rem:dist_vs_newdist}
  The extended ``metric'' $\newdist$ is not a metric, since it only satisfies
  the triangle inequality up to a bounded additive error, but this will not
  cause a problem.

  Notice that $\newdist$ agrees with $\dist$ on $G$, and also that the choice
  of vertices $x^\star$ and $y\star$ in the definition does not matter very
  much: for any $x$ and $y$ in $K$, if $x^\star$ and $y^\star$ are an
  alternative choice of vertices of simplices containing $x$ and $y$
  respectively, then
  \begin{equation*}
    \abs{\newdist(x,y) - \dist(x^\star, y^\star)} \leq 2D
  \end{equation*}
\end{remark}

We denote by $\gromov{a}{b}{c}$ the Gromov product of $b$ and $c$ in $G \union
\boundary G$ with respect to base point $a \in G$. Using $\newdist$ we extend
the Gromov product to $K$:

\begin{definition}
  Let $x \in K$ and let $y$ and $z$ be points in $K \union \boundary G$. Let
  $x^\star$ be a vertex of the minimal simplex of $K$ containing $x$. If $y$
  (respectively $z$) is in $\boundary G$ let $y^\star = y$ (respectively
  $z^\star = z$); otherwise let $y^\star$ (respectively $z^\star$) be a vertex
  of the minimal simplex of $K$ containing $y$ (respectively $z$.)
  Then define the \emph{Gromov product $\gromov{x}{y}{z}$ of $x$, $y$ and $z$}
  as $\gromov{x}{y}{z} = \gromov{x^\star}{y^\star}{z^\star}$.
\end{definition}

We now equip $K \union \boundary G$ with a topology as follows.

\begin{definition}
  We define a topology on $K \union \boundary G$ by describing a neighbourhood
  basis of each point.
  \begin{enumerate}
    \item For $x \in K$, the neighbourhoods of $x$ are the neighbourhoods of
      $x$ in $K$, treating $K$ with the usual topology of a simplicial complex.
    \item For $\xi \in \boundary G$, the set
      \begin{equation*}
        \set{x \in K \union \boundary G 
              \suchthat \gromov{e}{x}{\xi} \geq N}_{N\in\naturals}
      \end{equation*}
      is a fundamental system of neighbourhoods for the point $x$.
  \end{enumerate}
\end{definition}

Finally, we endow $\boundary G$ with a visual metric in the usual way.

\begin{definition}
  Let $\rho$ be a visual metric on $\boundary G$, so that there exist constants
  $k_1 \in (0,1]$, $k_2 \geq 1$ and $a > 1$ such that for any $\xi_1$ and
  $\xi_2$ in $\boundary G$,
  \begin{equation*}
    \rho(\xi_1, \xi_2) \in \left[k_1 a^{-\gromov{e}{\xi_1}{\xi_2}}, k_2
          a^{-\gromov{e}{\xi_1}{\xi_2}}\right]
  \end{equation*}
\end{definition}

\subsection{An inverse system of complexes}

We now define a sequence of subcomplexes in $K$. We will use the topology of
these complexes to approximate the topology of $\boundary G$. 

\begin{definition}
  For $n \in \naturals$ let $S_n$ be the sphere in $G$ of radius $n$, i.e.\ the 
  set $\set{g \in G \suchthat \dist(e, g) = n}$ of elements of $G$ whose 
  distance from the identity element is $n$. 

  Let $K_n$ be the full subcomplex of the Rips complex $K$ with vertex set
  $S_n$.
\end{definition}

\begin{definition}
  Fix an ordering on $S$. Let $S^\star$ be the free monoid on 
  $S$; order $S^\star$ by the lex-least ordering: this ordering is defined by 
  requiring that shorter words precede longer and that words of equal length are 
  ordered lexicographically. 

  For $n \geq m$ define a map $p^n_m \from S_n \to S_m$ as follows.
  Given an element $g$ of $S_n$ let $w_g \in S^\star$ be the lex-least
  representative of $g$; note that this word in $S^\star$ necessarily has
  length $n$.  Then define $p^n_m(g)$ to be the element of $G$ given by
  truncating $w_g$ to a word of length $m$.
\end{definition}

\begin{lemma}\label{lem:psimplicial}
  For $n - m > D + \delta$, the map $p^n_m$ defines a simplicial map from
  $K_n \to K_m$.
\end{lemma}

\begin{proof}
  Let $\set{g_1, \dotsc, g_k}$ span a simplex in $K_n$.  Then for $i$ and
  $j$ in $\set{1, \dotsc, k}$, the distance from $p^n_m(g_i)$ to any
  geodesic from $g_i$ to $g_j$ is at least $n - m - D$, which is greater than
  $\delta$, so by hyperbolicity of a triangle with vertices $e$, $g_i$ and $g_j$,
  the distance from $p^n_m(g_i)$ to any geodesic from $e$ to $g_j$ is at most
  $\delta$. We may choose this geodesic to contain $p^n_m(g_j)$, and $\dist(e,
  p^n_m(g_i)) = \dist(e, p^n_m(g_j))$, so $\dist(p^n_m(g_i), p^n_m(g_j)) \leq
  2\delta \leq D$.  It follows that the diameter of $\set{p^n_m(g_1), \dotsc,
  p^n_m(g_k)}$ is at most $D$ and therefore the set spans a simplex in $K_m$.
\end{proof}

\begin{remark}
  Note that $p^m_l \composed p^n_m = p^n_l$ whenever these maps are defined, and 
  therefore $(\set{K_n}_n, \{p^n_m\}_{n - m > D + \delta})$ is an inverse system 
  of simplicial complexes.
\end{remark}

\subsection{Projecting from $\boundary G$}

In order to define projections from $\boundary G$ into our system $(K_n)$, we
first describe a sequence of covers of $\boundary G$.

\begin{definition}
  For $x \in S_n$ let $U_n(x)$ be the set of limit points in $\boundary G$ of 
  geodesic rays $\gamma$ with $\gamma(0) = e$ and $\dist(\gamma(n), x) \leq 
  2\delta+1$.  
  
  By~\cite[III.H.3.6]{bridsonhaefliger99} $U_n(x)$ is a (not necessarily open)
  neighbourhood in $\boundary G$ of the set of limit points in $\boundary G$ of
  geodesic rays $\gamma$ with $\gamma(0) = e$ and $\gamma(n) = x$.
\end{definition}

\begin{lemma}\label{lem:L_Ksimplicial}
  The nerve of $\{U_n(x)\}_{x \in S_n}$ is a subcomplex of $K_n$.
\end{lemma}

\begin{proof}
  Let $g_1, \dotsc, g_k \in S_n$ span a simplex in the nerve of
  $\{U_n(x)\}_{x \in S_n}$, so there exists a point $\xi \in
  \bigintersection_{i=1}^k U_n(g_i)$.  Let $\gamma_1, \dotsc, \gamma_n$ be
  geodesic rays with $\gamma_i(0) = e$, $\gamma_i(\infty) = \xi$ and
  $\dist(\gamma_i(n), g_i) \leq 2\delta+1$.  Then $\Diam\set{\gamma_i(n)
  \suchthat i=1, \dotsc, n} \leq 2\delta$, and it follows that
  \begin{equation*}
    \Diam\{g_1, \dotsc, g_k\} \leq 6\delta+ 2 \leq D.\qedhere
  \end{equation*}
\end{proof}

\begin{definition}
  For each $n$, fix a partition of unity subordinate to the covering
  $\set{U_n(x)}_{x \in S_n}$ of $\boundary G$. Equivalently, this is a
  continuous map from $\boundary G$ to the nerve of the covering. Denote by
  $p^\infty_n\from \boundary G \to K_n$ the composition of this map with the
  inclusion of the nerve into $K_n$, as given by Lemma~\ref{lem:L_Ksimplicial}.
\end{definition}

\begin{lemma}\label{lem:close_projections}
  Let $\xi \in \boundary G$. Then $\newdist(p^\infty_n(\xi),
  p^\infty_{n+1}(\xi)) \leq 6\delta+3 + 2D$.
\end{lemma}

\begin{proof}
  Choose vertices $x_1$ and $x_2$ respectively to be vertices of minimal
  simplices of $K_n$ and $K_{n+1}$ containing $p^\infty_n(\xi)$ and
  $p^\infty_{n+1}(\xi)$. Then for $i$ equal to $1$ or $2$ there is a geodesic
  ray $\gamma_i$ with $\gamma_i(0) = e$, $\gamma_i(\infty) = \xi$, and so that
  $\dist(\gamma_1(n), x_1) \leq 2\delta+1$ and $\dist(\gamma_2(n+1),
  x_2) \leq 2\delta+1$.

  Then let $y_1 = \gamma_1(n)$ and $y_2 = \gamma_2(n+1)$. By considering the
  degenerate triangle with edges $\gamma_1$ and $\gamma_2$,
  $\dist(y_1, \gamma_2) \leq \delta$, so by the triangle inequality
  $\dist(y_1, y_2) \leq 2\delta +1$. It follows that $\dist(x_1, x_2)
  \leq 6\delta+3$, and so the claimed inequality holds by
  \cref{rem:dist_vs_newdist}.
\end{proof}

\begin{proposition}\label{prop:projectstoasimplex}
  Let $U$ be a subset of $\boundary G$ with diameter $\rho_0$ with respect to
  the visual metric $\rho$. Let $n \leq \log_a(k_1/\rho_0)$. Then
  $p^\infty_n(U)$ is contained in a single simplex of $K_n$
\end{proposition}

To prove this proposition we first prove the following lemma.

\begin{lemma}\label{lem:boundary_gromov_product}
  Let $\alpha_1$ and $\alpha_2$ be geodesic rays with $\alpha_1(0) =
  \alpha_2(0) = e$. Then for any $m_1$ and $m_2$,
  \begin{equation*}
    \gromov{e}{\alpha_1(m_1)}{\alpha_2(m_2)} \geq 
        \min\{m_1 - 3\delta, m_2 - 3\delta, \gromov{e}{\alpha_1(\infty)}{\alpha_2(\infty)} - 6\delta\}
  \end{equation*}
\end{lemma}

\begin{proof}
  By~\cite[III.H.3.17]{bridsonhaefliger99},
  \begin{equation*}
    \gromov{e}{\alpha_1(\infty)}{\alpha_2(\infty)} \leq \liminf_{n_1, n_2}
        \gromov{e}{\alpha_1(n_1)}{\alpha_2(n_2)} + 2\delta.
  \end{equation*}
  Let $n_1 \geq m_1$ and $n_2 \geq m_2$. 
  
  Let $p_1$ be a point on $[\alpha_1(n_1), \alpha_2(n_2)] \union \alpha_2$
  within a distance $\delta$ of $\alpha_1(m_1)$.  If $p_1$ is on $\alpha_2$
  then
  \begin{align*}
    \dist(\alpha_1(m_1), \alpha_2(m_2)) 
      & \leq \delta + \dist(p_1, \alpha_2(m_2)) \\
      & \leq 2\delta + \abs{m_1 - m_2}
  \end{align*}
  It follows that
  \begin{align*}
    \gromov{e}{\alpha_1(m_1)}{\alpha_2(m_2)} 
      &\geq \frac{1}{2}(m_1 + m_2 - 2\delta - \abs{m_1 - m_2})\\
      &\geq \min\{m_1, m_2\} - \delta
  \end{align*} 
  and the result follows. Otherwise $p_1$ is on $[\alpha_1(n_1),
  \alpha_2(n_2)]$. Similarly either the claimed inequality holds, or there is a
  point $p_2$ on $[\alpha_1(n_1), \alpha_2(n_2)]$ within a distance $\delta$ of
  $\alpha_1(m_2)$.
  
  If $p_1$ lies between $\alpha_2(n_2)$ and $p_2$ on $[\alpha_1(n_1),
  \alpha_2(n_2)]$ then, by considering the triangle with vertices $p_2$,
  $\alpha_2(m_2)$ and $\alpha_2(n_2)$, we see that $\dist(p_1, \alpha_2) \leq
  2\delta$. Therefore $\dist(\alpha_1(m_1), \alpha_2) \leq 3\delta$, and so by
  a repetition of the argument used earlier,
  \begin{equation*}
    \gromov{e}{\alpha_1(m_1)}{\alpha_2(m_2)} \geq \min\{m_1, m_2\} - 3\delta
  \end{equation*}
  and the result follows.

  Otherwise,
  \begin{equation*}
    \dist(\alpha_1(n_1), \alpha_2(n_2)) = \dist(\alpha_1(n_1), p_1) +
        \dist(p_1, p_2) + \dist(p_2, \alpha_2(n_2)).
  \end{equation*}
  Therefore,
  \begin{align*}
    \mathrlap{\gromov{e}{\alpha_1(n_1)}{\alpha_n(n_2)} - \gromov{e}{\alpha_1(m_1)}{\alpha_2(m_2)}}
    \quad\quad\phantom{{}={}}&\\
           &\mathllap{{}={}}\left(\dist(\alpha_1(m_1), \alpha_1(n_1)) - \dist(\alpha_1(m_1), p_1)\right) \\
           &+ \left(\dist(\alpha_1(m_1), \alpha_2(m_2)) - \dist(p_1, p_2)\right) \\
           &+ \left(\dist(\alpha_2(m_2), \alpha_2(n_2)) - \dist(p_2, \alpha_2(m_2))\right) \\
           &\mathllap{{}\leq{}} \delta+2\delta+\delta=4\delta.
  \end{align*}
  It follows that
  \begin{equation*}
    \gromov{e}{\alpha_1(\infty)}{\alpha_2(\infty)} \leq \gromov{e}{\alpha_1(m_1)}{\alpha_2(m_2)} + 6\delta,
  \end{equation*}
  which completes the proof.
\end{proof}

\begin{proof}[Proof of Proposition~\ref{prop:projectstoasimplex}] 
  Let $\xi_1$ and $\xi_2$ be points in $U$, so $\rho(\xi_1, \xi_2)
  \leq \rho_0$.  For $i$ equal to $1$ or $2$ let $x_i$ be a vertex of the
  minimal simplex of $K_n$ containing $p^\infty_n(\xi_i)$. Let $\alpha_i$ be a
  geodesic ray with $\alpha(0) = e$, $\alpha_i(\infty) = \xi_i$ and
  $\dist(\alpha_i(n), x_i) \leq 2\delta+1$. 
  
  By definition of the visual metric, $\gromov{e}{\xi_1}{\xi_2} \geq
  \log_a(k_1/\rho_0)$. Therefore, by \cref{lem:boundary_gromov_product},
  \begin{equation*}
    \gromov{e}{\alpha_1(n)}{\alpha_2(n)} \geq \min\{n - 3\delta, \log_a(k_1/\rho_0) - 6\delta\}.
  \end{equation*}
  In particular,
  \begin{equation*}
    \dist(\alpha_1(n), \alpha_2(n)) \leq \max\{6\delta, 12\delta + 2(n - \log(k_1/\rho_0))\}
  \end{equation*}
  and the result follows since $D \geq 12\delta$.
\end{proof}

\section{Horoballs in the Rips complex}\label{sec:horoball}

In Section~\ref{sec:double_dagger} we will use a limiting argument, which will
relate the topology of the complement of a point in the boundary $\boundary G$
to a particular subcomplex of $K$, which we will think of as a horoball in $K$.
In this section we prove that this subcomplex possesses the properties that we
will apply later.

\subsection{Defining the horoball}

For each $n$, let $v_n$ be a vertex of $K_n$ and let $g_n$ be an element of $G$
taken so that $g_n\cdot v_n$ is a point $v \in K_n$ independent of $n$.  Since
$K$ is locally finite, we can pass to a subsequence so that $v_i \in
K_{n_i}$ satisfies the following conditions:
\begin{enumerate}
  \item
    The translates by $g_i$ of large neighbourhoods of $v_i$ in the sphere
    $S_{n_i} \subset G$ are equal: for any $i \geq j$,
    \begin{equation*}
      N_j(v) \intersection g_i(S_{n_i}) = N_j(v) \intersection g_j(S_{n_j})
    \end{equation*}
    (Here $N$ denotes a neighbourhood taken with respect to the word metric on
    $G$.)
  \item
    The sequence $(g_i)_{i=1}^\infty$ of group elements, considered also as a
    sequence of vertices of $K$, converges to a point $\xi \in \boundary G$ as
    $i\to\infty$.
\end{enumerate}

Let $H_i$ be the full subcomplex of $K$ with vertex set $N_i(v) \intersection
g_i(S_{n_i})$; notice that $H_i$ contains $H_j$ for $i \geq j$, and also that
$g_i$ maps the full subcomplex of $K_{n_i}$ with vertex set $N_i(v_i)
\intersection S_{n_i}$ isometrically onto $H_i$. 

Let $H$ be the union of the subcomplexes $H_i \subset K$. We view $H$ as a
horosphere, the topology of which approximates that of $\boundary G - \set\xi$.
The rest Section of~\ref{sec:horoball} will be spent justifying this point of view.

\subsection{Geodesics and the horosphere}

\begin{lemma}\label{lem:nonempty_intersection}
  Every bi-infinite geodesic with one end point equal to $\xi$ meets $H$.
  
  Furthermore, the intersection between the geodesic and $H$ is contained in a
  ball with centre $v$ and radius depending only on $\delta$ and the distance
  from the $v$ to the geodesic.
\end{lemma}

\begin{proof}
  Let $\alpha$ be a bi-infinite geodesic with $\xi$ as an end point.  Let the
  distance $\dist(v,\alpha)$ be realised by a point $x \in \alpha$.

  The points $g_i$ converge to $\xi$ as $i\to\infty$, so the distances
  $\dist(x, [\xi, g_i])$ tend to infinity as $i$ tends to infinity. The
  triangle with vertices $x$, $\xi$ and $g_i$ is $\delta$-slim, so for each $i$
  there exists a point $y_i$ on $[x, g_i]$ with $\dist(y_i, \alpha) \leq
  \delta$ and $y_i\to\xi$ as $i\to\infty$.  For each $i$ let $z_i$ be a point
  on $\alpha$ with $\dist(y_i, z_i) \leq \delta$.

  Since $\dist(g_i, v) = n_i$ for all $i$ and $\dist(x,y_i) \to
  \infty$ as $i\to\infty$, we have the following limit.
  \begin{align*}\label{eqn:zi_closer_to_gi_than_x}
    n_i - \dist(g_i, z_i) & \geq n_i - \dist(g_i, y_i) - \delta \\
                          &= n_i - \dist(g_i, x) + \dist(x, y_i) - \delta\\
                          &\geq n_i - \dist(g_i, v) - \dist(x, v) + \dist(x,y_i) - \delta\\
                          &= \dist(x,y_i) - \dist(x,v) - \delta\\
                          &\to \infty\text{ as $i \to \infty$}\tag{$\star$}
  \end{align*}
  In particular, $n_i > \dist(g_i, z_i)$ for $i$ sufficiently large, so for
  $i$ sufficiently large, there exists a point $p_i$ on $\alpha$
  such that $z_i$ lies between $\xi$ and $p_i$ on $\alpha$ and so that
  $\dist(g_i,p_i)=n_i$. We now show that $\dist(v, p_i)$ is uniformly bounded.

  The limit~\eqref{eqn:zi_closer_to_gi_than_x}, taken together with the fact
  that $\dist(g_i, p_i) = n_i$ and $\dist(g_i, x) \geq n_i - \dist(x, v)$,
  shows that the distances $\dist([z_i, g_i], x)$ and $\dist([z_i, g_i], p_i)$
  tend to infinity as $i\to\infty$. Let $i_0$ be so that both distances are
  greater than $\delta$ for $i\geq i_0$.

  We now assume $i \geq i_0$ and divide into cases according to which of $x$
  and $p_i$ is closer to $z_i$.
  \begin{enumerate}
    \item If $x$ is between $z_i$ and $p_i$ on $\alpha$ then by considering a
      geodesic triangle with vertices $g_i$, $z_i$ and $p_i$ we see that for
      $i$ sufficiently large, $x$ is within a distance $\delta$ of a point
      $q_i$ on the geodesic from $g_i$ to $p_i$.  Therefore,
      \begin{align*}
        \dist(g_i,x)&\leq \dist(g_i,q_i)+\delta\\
                        &\leq \dist(g_i,p_i)-\dist(p_i,q_i)+\delta\\
                        &\leq \dist(g_i,p_i)-\dist(x,p_i)+2\delta.
      \end{align*}
      It follows that $\dist(x,p_i) \leq \dist(g_i,p_i)-\dist(g_i,x)+2\delta
      \leq \dist(x, v) + 2\delta$.
    \item If $p_i$ is between $z_i$ and $x$ on $\alpha$ then by an identical 
      argument with the r\^oles of $x$ and $p_i$ reversed we can similarly 
      deduce that 
      \begin{equation*}
        \dist(x,p_i)\leq\dist(g_i,x) - \dist(g_i,p_i) + 2\delta \leq  \dist(x,v) + 2\delta.
      \end{equation*}
  \end{enumerate}

  Either way, $\dist(v, p_i) \leq 2\dist(x, v) + 2\delta$ for $i \geq i_0$. In
  summary, we have shown that for $i$ sufficiently large there is a point
  $p_i\in\alpha$ with $\dist(g_i,p_i)=n_i$ and $\dist(v, p_i)$ bounded by
  $2\dist(x,w) + 2\delta$. Then $p_i \in H_i \intersection \alpha$ whenever $i$
  is at least $2\dist(x,w) + 2\delta$. Therefore the intersection is non-empty. 

  To see that it is bounded, let $p \in H \intersection \alpha$. Since $z_i \to
  \xi$ as $i\to \infty$, $z_i$ is between $\xi$ and $p$ on $\alpha$ for $i$
  sufficiently large. Fix $i \geq i_0$ so that this is true and also so that $p
  \in H_i$, so $\dist(g_i, p) = n_i$. Then the point $p$ has the properties
  assumed of $p_i$ above, so the argument above shows that $\dist(p, v) \leq
  2\dist(x,v) + 2\delta$. Therefore the intersection $H \intersection \alpha$
  is bounded.
\end{proof}

\begin{lemma}\label{lem:bounded_intersection}
  Let $C$ be a compact subset of $\boundary G - \set\xi$. Then the set of points 
  of intersection between $H$ and bi-infinite geodesics with one end point equal 
  to $\xi$ and the other in $C$ is bounded.
\end{lemma}

\begin{proof}
  By Lemma~\ref{lem:nonempty_intersection} it is sufficient to show that
  $\dist(v, \alpha)$ is bounded as $\alpha$ ranges over all bi-infinite
  geodesics with one end point equal to $\xi$ and the other in $C$.

  Let $\beta$ be a geodesic ray from $v$ to $\xi$. Suppose that there is a
  sequence $\zeta_i$ of points in $C$ with $\dist(v, [\xi, \zeta_i]) = m_i$
  where $m_i \to \infty$ as $i\to \infty$. Let $\beta_i$ be a geodesic ray from
  $v$ to $\zeta_i$. Let $p_i$ be a point on $\beta$ with $\dist(p_i, v) = m_i -
  \delta - 1$, so $\dist(p_i, [\xi, \zeta_i]) > \delta$. Then by hyperbolicity
  of the triangle with edges $\beta$, $\beta_i$ and $[\xi, \zeta_i]$ there is a
  point $q_i \in \beta_i$ with $\dist(p_i, q_i) = \delta$. Then $\dist(v, q_i)
  \geq \dist(v, p_i) - \delta$. Therefore,
  \begin{align*}
    \gromov{v}{\xi}{\zeta_i} &\geq \gromov{v}{p_i}{q_i} \\
                             &\geq \dist(v, p_i) - \delta\\
                             &\to \infty \text{ as $i\to\infty$},
  \end{align*}
  so $\zeta_i\to\xi$ as $i\to\infty$. But $C$ is a compact subset of $\boundary
  G - \set\xi$, so we reach a contradiction.
\end{proof}

\subsection{Projections from the boundary to the horosphere}

\begin{definition}
  For each $i$, let $U_i \subset \boundary G$ be the set $\left(g_i \composed
  p^\infty_{n_i} \composed g_i^{-1}\right)^{-1} H_i$ and let $q_i \from U_i \to
  H_i$ be the restriction of $g_i \composed p^\infty_{n_i}\composed g_i^{-1}$
  to $U_i$.
\end{definition}

The maps $q_i$ are described by the following lemma.

\begin{lemma}\label{lem:describing_q_i}
  Let $\zeta \in \boundary G$ be a point such that
  \begin{equation*}
    N_{2\delta+1}^{g_i\cdot S_{n_i}} \set{\eta(n_i)\suchthat\text{$\eta$ a geodesic ray with
        $\eta(0) = g_i$ and $\eta(\infty) = \zeta$}}
  \end{equation*}
  is contained in $H_i$, where $N_{2\delta+1}^{g_i\cdot S_{n_i}}$ denotes a
  $2\delta+1$-neighbourhood in $g_i\cdot S_{n_i}$ taken with respect to the
  word metric $\dist$.  Then $\zeta \in U_i$ and $q_i(\zeta)$ is contained in
  the simplex of $H_i$ spanned by this set of vertices.
\end{lemma}

\begin{proof}
  This is immediate from the definition of the maps $p^\infty_{n_i}$. Note that
  the set has diameter at most $6\delta+1$, so does indeed span a simplex.
\end{proof}

\begin{lemma}
  Every compact subset of $\boundary G - \set\xi$ is contained in $U_i$ for 
  $i$ sufficiently large.
\end{lemma} 

\begin{proof}
  Suppose that a compact subset $C$ of $\boundary G - \set\xi$ is not contained
  in $U_i$ for infinitely many $i$. Then after passing to a subsequence there
  exists for each $i$ a point $\zeta_i \in C\setminus U_i$, so by
  Lemma~\ref{lem:describing_q_i} for each $i$ there is a geodesic ray
  $\alpha_i$ from $g_i$ to $\zeta_i$ that meets $N_{2\delta+1}^{g_i\cdot
  S_{n_i}}(g_i\cdot S_{n_i} - H_i)$, and so does not meet $H_i\intersection
  N_{i-4\delta-1}(v)$, which is equal to $H \intersection N_{i-4\delta-1}(v)$.
  After passing to a subsequence and reparametrising, the rays $\alpha_i$
  converge uniformly on compact subsets to a bi-infinite geodesic from $\xi$ to
  some point $\zeta \in C$ that does not meet $H$, contradicting
  Lemma~\ref{lem:nonempty_intersection}.
\end{proof}

\begin{lemma}\label{lem:affinehomotopic}
  For any compact subset $C\subset\boundary G - \set\xi$, there exists $i_0$ 
  such that for all $i \geq i_0$ and any $\zeta \in C$, $\newdist(q_i(\zeta),
  q_{i_0}(\zeta)) \leq 10\delta + 2$, and, in particular, $q_i\restricted{C}$
  and $q_{i_0}\restricted{C}$ are homotopic.
\end{lemma}

\begin{proof}
  Let $B\subset H$ be the set of points of intersection between $H$ and
  geodesics with one end point equal to $\xi$ and the other in $C$; this set is
  bounded by Lemma~\ref{lem:bounded_intersection}.

  Fix $\alpha$ to be some geodesic with $\alpha(-\infty)=\xi$ and
  $\alpha(\infty)\in C$ and let $x$ be a point of intersection between $\alpha$
  and $H$, so $x \in B$.  Since the triangle with vertices $x$, $g_i$ and $\xi$
  is $\delta$-slim, and $\dist(x, [\xi, g_i]) \to \infty$ as $i\to\infty$, for
  each $i$ there exists a point $y_i$ on $[g_i, x]$ with $\dist(y_i, \alpha)
  \leq \delta$ and so that $y_i \to \xi$ as $i\to\infty$. Let $z_i \in \alpha$
  be such that $\dist(y_i, z_i)\leq\delta$.
  
  Now fix $i_0$ such that the following conditions hold for all $i\geq 
  i_0$.
  \begin{enumerate}
    \item The distance $\dist(y_i,x)$ is greater than $7\delta + \Diam B$.
    \item The set $C$ is contained in $U_i$.
  \end{enumerate}

  Let $\zeta \in C$ and let $\alpha'$ be a geodesic from $\xi$ to $\zeta$. Let 
  $x'$ be a point in the intersection of $\alpha'$ and $B$. 

  Now let $i\geq i_0$. Then by considering the ideal triangle with vertices
  $g_i$, $x$ and $x'$, and using the fact that $\dist(y_i, x) > \delta + \Diam
  B$, we deduce that there exists a point $y_i'$ on $[x', g_i]$ with
  $\dist(y_i, y_i') \leq \delta$.

  Similarly, since $\dist(z_i, x) > \delta + \Diam B$, consideration of the
  triangle with vertices $\xi$, $x$ and $x'$ reveals that there exists a point
  $z_i'$ on $\alpha'$ between $\xi$ and $x'$ with $\dist(z_i, z_i') \leq
  \delta$. Then $\dist(y_i', z_i') \leq 3\delta$.

  Since $\dist(z_i', x') > 2\delta$, the distance from $y_i'$ to the sub-ray of
  $\alpha'$ from $x'$ to $\zeta$ is greater than $\delta$. Therefore by
  considering a triangle with vertices $g_i$, $x'$ and $\zeta$ we see that
  there exists a point $w_i$ on a geodesic ray $\eta_i$ from $g_i$ to $\zeta$
  with $\dist(w_i, y_i') \leq \delta$. It follows that $\dist(w_i, z_i') \leq
  4\delta$, and we know that $\dist(x', z_i') > 5\delta$, so $\dist(x', [z_i',
  w_i]) > \delta$ and therefore $\dist(x', \eta_i) \leq \delta$.

  In summary, any geodesic ray from $g_i$ to $\zeta$ passes within a distance 
  $\delta$ of $x'$, for $i\geq i_0$. Furthermore, $\dist(g_i,x') = n_i$, so 
  $\dist(\eta_i(n_i), x')\leq 2\delta$. It follows from the description of
  $q_i$ in \cref{lem:describing_q_i} that $q_i(\zeta)$ is contained in a
  simplex with all vertices within a distance $4\delta+1$ of $x'$. The set of
  all such vertices spans a simplex, so, in particular, $q_i\restricted{C}$ and
  $q_{i_0}\restricted{C}$ are homotopic for $i \geq i_0$. 
\end{proof}

\begin{lemma}\label{lem:propermaps} 
  For any compact subset $C \subset H$ there exists $i_0$ such that
  \begin{equation*}
    \bigunion_{i\geq i_0} q_i^{-1} C
  \end{equation*}
  is a relatively compact subset of $\boundary G - \set\xi$.
\end{lemma}

\begin{proof}
  If not then there exists a sequence $\zeta_i$ of points with $\zeta_i \in U_i$ 
  for all $i$ such that $q_i(\zeta_i) \in C$ for all $i$ and $\zeta_i \to \xi$ 
  as $i \to \infty$.

  Then for each $i$ a geodesic from $g_i$ to $\zeta_i$ passes within a bounded
  distance of $C$. After passing to a subsequence these geodesic rays converge
  to a bi-infinite geodesic passing within a bounded distance of $C$, both end
  points of which must be $\xi$. This is a contradiction.
\end{proof}

\begin{lemma}\label{lem:uniform_continuity}
  For any point $\zeta \in \boundary G - \set\xi$, there is a neighbourhood $V$
  of $\zeta$ and a number $i_0$ such that, for all $i \geq i_0$, $V \subset U_i$ 
  and the diameter of $q_i(V)$ with respect to $\newdist$ is at most $6\delta+2$.
\end{lemma}

\begin{proof}
  Let $\zeta \in \boundary G - \set\xi$ and let $\alpha$ be a bi-infinite
  geodesic from $\xi$ to $\zeta$. Let $x$ be a point on $\alpha$ so that the
  distance between the sub-ray of $\alpha$ from $x$ to $\zeta$ and the set $H$
  is at least $9\delta+2$.

  Let $V$ be the following neighbourhood of $\zeta$.
  \begin{equation*}
    V = \set{\beta(\infty) \suchthat \text{$\beta$ a geodesic with
            $\beta(-\infty) = \xi$ and $\dist(\beta, x) \leq 3\delta+1$}}.
  \end{equation*}
  The neighbourhood $V$ is relatively compact in $\boundary G - \set\xi$ so
  we may take $i_0$ such that $V \subset U_i$ for $i \geq i_0$. Enlarge $i_0$
  so that for any $i \geq i_0$, $\dist([\xi, g_i], x) \geq 4\delta+1$.

  Now let $\zeta_1$ and $\zeta_2$ be points in $V$ and let $i \geq i_0$. For
  $j$ equal to $1$ or $2$ let $\gamma_j$ be any geodesic ray from $g_i$ to
  $\zeta_j$ and let $\beta_j$ be a bi-infinite geodesic with $\beta_j(-\infty) =
  \xi$, $\beta_j(\infty) = \zeta_j$ and $\dist(\beta_j, x) \leq 3\delta+1$; let
  this minimal distance be realised by a point $y_j \in \beta_j$. By
  considering a triangle with vertices $\xi$, $g_i$ and $\zeta_j$, $\dist(y_j,
  \gamma_j) \leq \delta$, so $\dist(\gamma_j, x) \leq 4\delta+1$; let $z_j \in
  \gamma_j$ minimise this minimal distance. Then $\dist(z_1, z_2) \leq
  8\delta+2$.

  Now, $\dist(x, H) \geq 9\delta+2$, so $\dist(z_j, H) \geq 5\delta+1$.
  Furthermore, $\gamma_j(n_i) \in H$. It follows that $\dist(\gamma_j(n_i), [z_1,
  z_2]) > \delta$ and so by considering a triangle with vertices $g_i$, $z_1$
  and $z_2$, $\dist(\gamma_1(n_i), \gamma_2(n_i)) \leq 2\delta$. The result
  follows by lemma~\ref{lem:describing_q_i}. 
\end{proof}

\section{The double dagger condition}\label{sec:double_dagger}

We begin this section by recalling work of Bestvina and Mess, Bowditch, and
Swarup, which characterise the hyperbolic groups with locally path connected
boundary. In~\cite{bestvinamess91}, Bestvina and Mess introduce the following
condition that a hyperbolic group might satisfy. (When comparing the constants
appearing in our definition of the condition with that in the cited paper, note
that with our convention that ideal triangles are $\delta$-hyperbolic, Bestvina
and Mess's constant $C$ can be taken to be equal to $\delta$.)

\begin{definition}
  We say that $G$ satisfies $\ddag_M$ if there exists a number $L$ such that,
  for any $n$ and any $x$ and $y$ in $S_n$ with $\dist(x,y) \leq M$, $x$ and
  $y$ are connected by a path of length at most $L$ in the complement of the
  ball with radius $\dist(x,y) - \delta$ with centre $1 \in G$.
\end{definition}

This property is linked to the connectedness of $\boundary G$ by the following
propositions:

\begin{proposition}\cite[Proposition 3.2]{bestvinamess91}
  If $G$ satisfies $\ddag_{8\delta + 3}$ then $\boundary G$ is locally path
  connected.
\end{proposition}

\begin{proposition}\cite[Proposition 3.3]{bestvinamess91}\label{prop:BM2}
  If $\boundary G - \set\xi$ is connected for all $\xi \in \boundary G$ then
  $G$ satisfies $\ddag_M$ for any $M$.
\end{proposition}

Taken together with the following theorem, first proved by
Bowditch~\cite{bowditch98b} when $G$ is strongly accessible and then by
Swarup~\cite{swarup96} in the general case, these propositions show that
$\boundary G$ is locally path connected if and only if it is connected.

\begin{theorem}\cite{bowditch98b,swarup96}
  If $\boundary G$ is connected then $\boundary G$ does not contain a cut
  point.
\end{theorem}

\subsection{A modified condition}

In this section we introduce a related condition. This condition appears
stronger than $\ddag_M$, but, using Bestvina and Mess's result, we show that it
holds under the same assumption. This condition reinterprets the paths in
$\ddag_M$ as edge paths in the complexes $K_n$, with the additional assumption
that these paths approximate paths that factor through the projection map
$p^\infty_n$. This interpretation strengthens the link between
path-connectedness of the boundary and connectedness of the complexes $K_n$.

\begin{definition}
  Let $I$ be the interval $[0,1]$, which we identify with a 1-simplex.  

  We say that $G$ satisfies $\ddag'_M$ if there exist numbers $L$ and $n_0$
  such that, for any $n \geq n_0$, any $\zeta_1$ and $\zeta_2$ in $\boundary G$
  such that $\newdist(p^\infty_n(\zeta_1), p^\infty_n(\zeta_2)) \leq M$, and
  any vertices $x_1$ and $x_2$ of simplices of $K_n$ containing
  $p^\infty_n(\zeta_1)$ and $p^\infty_n(\zeta_2)$, there is a path $\alpha\from
  I\to \boundary G$ from $\zeta_1$ to $\zeta_2$ in $\boundary G$ such that
  $p^\infty_n\composed\alpha$ admits a simplicial approximation connecting
  $x_1$ to $x_2$ in $K_n$ after at most $L$ barycentric subdivisions of $I$.
\end{definition}

\begin{proposition}\label{prop:sufficient_for_ddag_p}
  If $\boundary G - \set\xi$ is path connected then $G$ satisfies $\ddag'_M$
  for any $M$.
\end{proposition}

\begin{proof}
  Suppose $\ddag'_M$ does not hold. Then for any $i$ there exists a number $n_i
  \geq i$, points $\zeta_1^i$ and $\zeta_2^i$ in $\boundary G$ and vertices
  $x_1^i$ and $y_1^i$ of simplices of $K_{n_i}$ containing
  $p^\infty_{n_i}(\zeta^i_1)$ and $p^\infty_{n_i}(\zeta^i_2)$ such that
  $\newdist(p^\infty_{n_i}(\zeta^i_1), p^\infty_{n_i}(\zeta^i_2)) \leq M$ and
  there is no path $\alpha$ from $\zeta_1^i$ to $\zeta_2^i$ in $\boundary G$
  where $p^\infty_{n_i}\composed\alpha$ admits a simplicial approximation after
  at most $i$ barycentric subdivisions of the interval.

  Take group elements $g_i$ such that $g_i\cdot x_1^i$ is equal to some vertex
  $x_1$ for all $i$ vertices and pass to a subsequence as in
  Section~\ref{sec:horoball}, so that $g_i\to\xi$ as $i \to\infty$. Since
  $\dist(x_1^i, x_2^i)$ is bounded we can pass to a deeper subsequence so that
  $g_i\cdot x_2^i$ is constant equal to some vertex $x_2$ for all $i$, and
  $x_2$ is contained in $H_i$ for all $i$. 

  Now, $g_i\cdot \zeta_1^i$ and $g_i\cdot\zeta_2^i$ map by $q_i$ to points in
  simplices containing $x_1$ and $x_2$, so by Lemma~\ref{lem:propermaps} we
  can pass to a subsequence so that the sequences $g_i\cdot \zeta_1^i$ and
  $g_i\cdot\zeta_2^i$ converge to points $\zeta_1$ and $\zeta_2$ in $\boundary
  G - \set\xi$ as $i \to \infty$.
  
  Let $C$ be a compact subset of $\boundary G - \set\xi$ containing a path
  $\beta$ from $\zeta_1$ to $\zeta_2$ and also a neighbourhood of each of the
  points $\zeta_1$ and $\zeta_2$. By Lemma~\ref{lem:uniform_continuity}, there
  exists $\epsilon > 0$ and $i_0 \in \naturals$ such that for subset $V$ of
  $C$ of diameter at most $\epsilon$ and any $i \geq i_0$, $q_i(V)$ has
  diameter (with respect to $\newdist$) at most $6\delta+2$. By increasing
  $i_0$ we may also assume that $\zeta_1^i$ and $\zeta_2^i$ are contained in
  $C$ for $i \geq i_0$. By decreasing $\epsilon$ we may assume that the
  $\epsilon$-neighbourhoods of $\zeta_1$ and $\zeta_2$ are contained in $C$. 

  Since $\boundary G$ is locally path connected, for $j$ equal to $1$ or $2$
  there exists a neighbourhood of $\zeta_j$ such that any point in this
  neighbourhood is joined to $\zeta_j$ by a path contained in the
  $\epsilon$-neighbourhood of $\zeta_j$. Now increase $i_0$ to ensure that
  $\zeta_1^i$ and $\zeta_2^i$ are contained in these neighbourhoods for $i \geq
  i_0$. For $j=1$ and $j=2$ let $\alpha_1^j$ be a path in the ball of radius
  $\epsilon$ centred at $\zeta_j$ connecting $\zeta_j^i$ to $\zeta_j$.

  Then the concatenation $\alpha^i\from[0,1]\to \boundary G$ of $\alpha_1^i$,
  $\beta$ and $\alpha_2^i$ gives a path in $C$ from $\zeta_1^i$ to $\zeta_2^i$
  such that $I$ can be covered by boundedly many open intervals, each of which
  has image under $\alpha^i$ with diameter at most $\epsilon$. It follows that
  the composition of $\alpha^i$ with $q_i$ is a path in $H$ that admits a
  simplicial approximation connecting $x_1$ to $x_2$ after boundedly many
  subdivisions of the interval. For $i \geq i_0$ large enough this gives a
  contradiction.
\end{proof}

Using this property we define maps $i^n_{n+1}\from \Skel_1 K_n$ to
$\Skel_1K_{n+1}$ as follows.  We will make use of the following lemma, adapted
from from~\cite{bestvinamess91}.

\begin{lemma}\cite{bestvinamess91}\label{lem:near_geod_rays}
  Let $G$ be infinite. For any $x \in X$ there exists a geodesic ray $\alpha$ 
  with $\alpha(0) = e$ and $\dist(x,\alpha) \leq \delta$. 
  
  Furthermore, $\dist(x, \alpha(\dist(e,x))) \leq 2\delta$.
\end{lemma}

\begin{proof}
  For the first claim, see~\cite{bestvinamess91} and recall that ideal
  triangles are assumed to be $\delta$-hyperbolic. For the second, let $p$ be
  a point on $\alpha$ such that $\dist(x, p) \leq \delta$. Then the triangle
  inequality shows that $\abs{\dist(e,p) - \dist(e,x)} \leq \delta$, and so
  $\dist(p, \alpha(\dist(e,x))) \leq \delta$. Then the result follows from
  another application of the triangle inequality.
\end{proof}

\begin{proposition}\label{prop:i_in_dimension_1}
  Suppose that $G$ satisfies $\ddag'_{D + 4\delta+2}$. Then there exist $L$ and
  $n_0$ such that for $n\geq n_0$ there is a map $i^n_{n+1}\from \Skel_1 K_n
  \to \Skel_1 K_{n+1}$ satisfying the following conditions:
  \begin{enumerate}
    \item for each $n$, $i^n_{n+1}$ is a simplicial map on the subdivision
      $\sd^L\Skel_1 K_n$,
    \item for any vertex $x$ of $K_n$, $\dist(x, i^n_{n+1}(x)) \leq 2\delta+1$,
      and
    \item for each $n$, the map $i^n_{n+1}$ is a simplicial approximation a map
      that factors through $p^\infty_{n+1}\from \boundary G \to K_{n+1}$.
  \end{enumerate}
  (When $X$ is a simplicial complex, we denote by $\sd^n X$ the complex
  obtained from $X$ by barycentrically subdividing $n$ times.)
\end{proposition}

\begin{proof}
  Let $n_0$ and $L$ be the numbers given by the condition $\ddag'_{D + 4\delta
  + 2}$ and let $n \geq n_0$.

  First we define $i^n_{n+1}$ on the vertex set of $K_n$. For a vertex $x$ of
  $K_n$, let $\gamma_x$ be a geodesic with $\dist(x, \gamma_x(n)) \leq 2\delta$
  and let $i^n_{n+1}(x)$ be $\gamma_x(n+1)$, which is a vertex of a simplex
  containing $p^\infty_{n+1}(\gamma_n(\infty))$. Then $\dist(x, i^n_{n+1}(x))
  \leq 2\delta + 1$, as required by the second condition.

  For an edge $I$ of $K_n$ connecting vertices $x_1$ and $x_2$,
  $\dist(i^n_{n+1}(x_1), i^n_{n+1}(x_2))$ is at most $2(2\delta + 1) + D = D +
  4\delta + 2$. It follows from $\ddag'_{D+4\delta+2}$ that there is a path
  $\alpha$ from $\gamma_{x_1}(\infty)$ to $\gamma_{x_2}(\infty)$ in $\boundary
  G$ so that $p^\infty_{n+1}\composed\alpha$ admits a simplicial approximation
  connecting $i^n_{n+1}(x_1)$ to $i^n_{n+1}(x_2)$ after $L$ subdivisions.
  Define $i^n_{n+1}$ to be this simplicial map on $\sd^L I$.

  This process defines a map $i^n_{n+1}\from \Skel_1 K_n\to \Skel_1 K_{n+1}$
  with the required properties.
\end{proof}

\subsection{Extending maps from the boundary of a 2-simplex}

We now introduce a new condition, which should be thought of as a version of
$\ddag_M$ describing connectedness one dimension higher.

\begin{definition}
  Let $\Delta$ be the standard 2-simplex. 

  We say that $G$ satisfies $\S_M$ if there exist $L$ and $n_0$ such that for
  any $n\geq n_0$ and any simplicial map $f\from \sd^M\boundary\Delta \to K_n$
  that is a simplicial approximation to a map that factors through $p^\infty_n
  \from \boundary G \to K_n$, $f$ admits an extension to $\Delta$ that is
  affine on $\sd^L \Delta$. (We say that a map between simplicial complexes is
  affine if its restriction to any simplex of the domain is affine.)
\end{definition}

We now give a sufficient condition for $\S_M$ to be satisfied. This result is
analogous to Proposition~\ref{prop:BM2}.

\begin{proposition}\label{prop:sufficient_condition}
  Suppose that $\boundary G - \set\xi$ is simply connected for any point $\xi
  \in \boundary G$. Then $\S_M$ is satisfied for all $M$.
\end{proposition}

We will use the following lemma:

\begin{lemma}\label{lem:relative_approximation}
  Let $X$ be a simplicial complex and let $f\from\sd^M\boundary\Delta\to X$ be
  a simplicial map. Suppose that $f$ admits a continuous extension to $\Delta$.
  Then there exists $L$ such that $f$ admits an extension to $\Delta$ that is
  affine on $\sd^L\Delta$. 
\end{lemma}

\begin{proof}
  This follows from the main theorem of~\cite{zeeman64}: note that the degree
  $L$ subdivision of a complex can be obtained from any degree $L$ relative
  subdivision by further subdividing some simplices, so any map that is a
  simplicial map on the relative subdivision is affine on the full subdivision.
\end{proof}

\begin{proof}[Proof of Proposition~\ref{prop:sufficient_condition}]
  Suppose that the condition is not satisfied. Then for each $i$ there exists a
  number $n_i \geq i$, a map $\gamma_i \from \boundary\Delta \to \boundary G$
  and a simplicial approximation $f_i \from \sd^M \boundary\Delta \to K_{n_i}$
  to $p^\infty_{n_i}\composed\gamma_i$ such that $f_i$ does not admit an
  extension to $\Delta$ that is affine on $\sd^i\Delta$.

  Fix a vertex $v$ of $\boundary\Delta$; then there exist group elements $g_i
  \in G$ such that $g_i\cdot f_i(v)$ is a fixed vertex of $K$. Pass to a
  subsequence as in Section~\ref{sec:horoball}, and to a deeper subsequence so
  that $g_i\cdot f_i(w)$ is independent of $i$ for all vertices $w \in \sd^M
  \boundary\Delta$. Then the map $g_i\composed f_i$ is independent of $i$; call
  this map $f$. It remains to show that $f$ is null-homotopic in $H$; then it is
  null-homotopic by an affine map by Lemma~\ref{lem:relative_approximation},
  and this gives a contradiction for $i$ sufficiently large.

  By Lemma~\ref{lem:propermaps} there exists $i_0$ and a 
  compact set $C_0 \subset \boundary G - \set\xi$ such that $C_0$ contains the
  image of $\gamma_i$ for every $i \geq i_0$. By
  Lemma~\ref{lem:affinehomotopic} there exists $i_1 \geq i_0$ such that
  $q_i\restricted{C_0}$ is homotopic to $q_{i_1}\restricted{C_0}$ for every $i
  \geq i_1$.  
  
  Let $h$ be a contracting homotopy of $\gamma_{i_1}$ in $\boundary G -
  \set\xi$ and let $C_1$ be a compact subset of $\boundary G - \set\xi$
  containing $C_0$ and the image of $h$. Let $i_2 \geq i_1$ be so that $C_1
  \subset U_{i_2}$.

  Then $f$ is homotopic to $q_{i_1}\composed\gamma_{i_1}$, which is homotopic
  to $q_{i_2}\composed\gamma_{i_1}$, which is homotopic by $q_{i_2}\composed h$
  to a point.
\end{proof}

\section{Local simple connectedness}\label{sec:local_simple_connectedness}

Our goal in this section is to prove the following theorem.

\begin{restatable}{theorem}{ddagimpliesLCtwo}\label{thm:ddagimpliesLCtwo}
  If $G$ satisfies $\ddag'_M$ and $\S_M$ for all $M$ then $\boundary G$ is
  locally simply connected.
\end{restatable}

Throughout the rest of this section we assume that $G$ satisfies
$\ddag'_M$ and $\S_M$ for all $M$.

\subsection{Constructing maps to $\boundary G$}

We begin the proof of \cref{thm:ddagimpliesLCtwo} by proving the following 
proposition, which extends the maps constructed in
Proposition~\ref{prop:i_in_dimension_1}.

\begin{proposition}\label{prop:i_in_dimension_2}
  There are constants $n_0$ and $L$ such that for each $n\geq n_0$ there is map
  $i^n_{n+1}\from\Skel_2 K_n \to \Skel_2 K_{n+1}$ with the following properties.
  \begin{enumerate}
    \item For each $n$, $i^n_{n+1}$ is an affine map on the subdivision
      $\sd^L\Skel_2 K_n$.
    \item For any vertex $x$ of $K_n$, $\dist(x, i^n_{n+1}(x)) \leq 2\delta+1$.
  \end{enumerate}
\end{proposition}

\begin{proof} 
  Let $n_0$ be as in Proposition~\ref{prop:i_in_dimension_1} and for $n \geq
  n_0$ let $i^n_{n+1} \from \Skel_1 K_n \to \Skel_1 K_{n+1}$ be the maps given
  by that proposition, so there exists $M$ such that $i^n_{n+1}$ is a
  simplicial map on $\sd^L\Skel_1 K_n$, and this map is a simplicial
  approximation to a map that factors through $p^\infty_{n+1} \from \boundary G
  \to K_{n+1}$.

  Then $\S_L$ gives numbers $L'$ and $n_0' \geq n_0$ such that for $n \geq
  n_0'$, for any 2-simplex $\sigma$ of $K_n$, $i^n_{n+1}\restricted{\boundary
  \sigma}$ admits an extension to $\sigma$ that is affine on $\sd^{L'}\sigma$.
  Gluing these extensions together extends $i^n_{n+1}$ to a map $\Skel_2 K_n \to
  \Skel_2 K_{n+1}$ with the required properties.
\end{proof}

\begin{definition}
  With the notation of Proposition~\ref{prop:i_in_dimension_2}, for $n \geq m
  \geq n_0$ let $i^m_n \from \Skel_2 K_m \to \Skel_2 K_n$ be the composition
  $i^{n-1}_n \composed \dotsb \composed i^m_{m+1} \from \Skel_2 K_m \to 
  \Skel_2 K_n$.  This map is affine on the degree $(n-m)L$ barycentric
  subdivision of $\Skel_2 K_m$.
\end{definition}

In what follows we require the following useful lemma.

\begin{lemma}\label{lem:useful}
  Let $(x_n)$ be a sequence of points in $G$ with $\dist(x_{n+1}, e) = 
  \dist(x_n, e) + 1$ and $\dist(x_{n+1}, x_n)$ bounded. Then for $n \geq m$,
  \begin{equation*}
    \gromov{e}{x_m}{x_n} \geq \dist(e,x_m) - C
  \end{equation*}
  for some constant $C$ depending only on $\delta$ and the bound on 
  $\dist(x_{n+1}, x_n)$.
\end{lemma}

\begin{proof}
  By the triangle inequality, for any $n$ and $m$ we have $\dist(x_n, x_m) \geq
  \abs{n - m}$.  On the other hand, let $B$ be the bound on $\dist(x_{n+1},
  x_n)$. Then for any $n$ and $m$ we have $\dist(x_n, x_m) \leq B\abs{n - m}$. 
  Finally, $\dist(e, x_n) = \dist(e, x_0) + n$.
  
  It follows that the concatenation of a geodesic $[e, x_0]$ with the sequence
  $(x_n)_{n = 0}^\infty$ is a quasi-geodesic with constants depending only on
  $B$. Therefore by~\cite[Theorem III.H.1.7]{bridsonhaefliger99} there is a
  number $R$ depending only on $B$ and $\delta$ such that the sequence $(x_n)$
  is contained in an $R$-neighbourhood of a geodesic ray starting at $e$; let
  $(y_n)$ be a sequence of points on this geodesic with $\dist(y_n, x_n) \leq
  R$. Then $\gromov{e}{y_m}{y_n} = \min\{\dist(e, y_m), \dist(e, y_n)\}$, and
  the result follows with $C = 5R/2$.
\end{proof}

\begin{lemma}\label{lem:linearproduct}
  There is a constant $C$ such that if $x \in \Skel_2 K_m$ then
  \begin{equation*}
    \gromov{e}{i^m_n(x)}{x} \geq m - C
  \end{equation*}
  for all $n\geq m\geq n_0$.
\end{lemma}

\begin{proof}
  For $n \geq m$, let $y_n$ be a vertex of the minimal simplex of $K_n$
  containing $i^m_n(x)$. Then, for each $n \geq m$, we have $\dist(e, y_{n+1})
  = \dist(e, y_n)$.  We now bound the distance $\dist(y_n, y_{n+1})$. To start
  with, $\dist(y_n, i^n_{n+1} (y_n)) \leq 2\delta+1$. Let $R$ be $3\cdot
  2^{k-1}$, so that any two vertices in the 1-skeleton of $\sd^k \Delta$ are
  joined by a path traversing at most $R$ edges. Let $\sigma$ be the minimal
  simplex of $K_n$ containing $i^m_n(x)$. Then $y_n$ is a vertex of $\sigma$
  and $y_{n+1}$ is a vertex of the image of $\sigma$ under $i^n_{n+1}$, so
  $\dist(i^n_{n+1} y_n, y_{n+1}) \leq RD$.  Putting this together,
  \begin{equation*}
    \dist(y_n, y_{n+1}) \leq (2\delta + 1) + RD
  \end{equation*}

  It follows from \cref{lem:useful} that there exists a constant $C$ depending
  only on $\delta$ and $k$ such that $\gromov{e}{y_n}{y_m} \geq m -
  C$.  By applying the inequality in \cref{rem:dist_vs_newdist}, we deduce that
  \begin{equation*}
      \gromov{e}{i^m_n(x)}{x} \geq m - C - 3D.\qedhere
  \end{equation*}
\end{proof}

\begin{corollary}
  For $m \geq n_0$ and $x \in \Skel_2 K_m$, the sequence
  $(i^m_n(x))_{n\geq m}$ converges to a point in $\boundary G$, which we shall
  denote $i^m_\infty(x)$.
\end{corollary}

\begin{proof}
  For $n_1$ and $n_2$ at least $m$, \cref{lem:linearproduct} implies that
  $\gromov{e}{i^m_{n_1}(x)}{i^m_{n_2}(x)} \geq \min\{n_1,n_2\} - C$, which
  tends to infinity as $n_1$ and $n_2$ tend to infinity. Therefore
  $(i^m_n(x))_{n=m}^\infty$ converges to a point in $\boundary G$,
  by~\cite[Lemma III.H.3.13]{bridsonhaefliger99}.
\end{proof}

\begin{lemma}\label{lem:simplexproduct}
  There is a constant $C$ such that if $m \geq n_0$ and $x$ and $y$ are in a
  common simplex of $\Skel_2 K_m$ then
  \begin{equation*}
    \gromov{e}{i^m_n(x)}{i^m_n(y)} \geq m - C
  \end{equation*}
  for all $n\geq m$.
\end{lemma}

\begin{proof}
  The inequality of \cref{lem:linearproduct} is satisfied for each of $x$ 
  and $y$. By~\cite[Proposition III.H.1.22]{bridsonhaefliger99} there exist
  $\delta'$ depending only on $\delta$ such that
  \begin{align*}
    \gromov{e}{i^m_n(x)}{i^m_n(y)} &\geq \min\{\gromov{e}{i^m_n(x)}{x}, 
    \gromov{e}{x}{y}, \gromov{e}{y}{i^m_n(y)}\} - 2\delta'\\
    &\geq \min\{m - C, m - D/2, m - C\} - 2\delta' \\
    &\geq m - C'
  \end{align*}
  for some constant $C'$.
\end{proof}

The following corollary immediately follows from \cref{lem:simplexproduct}.

\begin{corollary}\label{cor:visualdiameter}
  There is a constant $C$ such that for any $m\geq n_0$ and any simplex
  $\sigma$ of $\Skel_2 K_m$, the diameter of $i^m_{\infty}(\sigma)$ with
  respect to the visual metric $\vdistance$ is bounded above by $Ca^{-m}$.
  \qed
\end{corollary}

\begin{proposition}\label{prop:continuouslimit}
  For $m\geq n_0$ the map $i^m_\infty\from \Skel_2 K_m \to \boundary G$ is
  continuous.
\end{proposition}

\begin{proof}
  Fix $x \in \Skel_2 K_m$. For $n\geq 0$ let $U_n$ be the minimal subcomplex of
  $\sd^{nk}\Skel_2 K_m$ containing a neighbourhood of $x$. Note that the
  collection $(U_n)_{n \geq 0}$ is a neighbourhood basis for $x$ in $\Skel_2
  K_m$.

  For each $n$, the restriction of $i^m_{m+n}$ to a simplex in 
  $\sd^{nk}\Skel_2K_m$ is affine, and therefore has image contained in a 
  simplex of $\Skel_2 K_{m+n}$. The point $x$ is contained in every maximal simplex 
  of $U_n$, so for any $y \in U_n$ the points $i^m_{m+n}(x)$ and $i^m_{m+n}(y)$ 
  are contained in a common simplex of $\Skel_2K_m$. Therefore, by 
  \cref{cor:visualdiameter}, there is a constant $C$ such that 
  $\vdistance(i^m_\infty(x), i^m_\infty(y)) \leq Ca^{-(m+n)}$ for every $n \geq
  0$ and every $y \in U_n$.
\end{proof}

\begin{remark}
  In fact, if a 2-simplex of $K_m$ is metrically identified with the standard
  2-simplex then the restriction of $i^m_\infty$ to this simplex can be shown
  to be H\"{o}lder continuous.  We do not require this strengthening of
  \cref{prop:continuouslimit} here.
\end{remark}

We complete our description of the maps $i^n_\infty$ with the following lemma,
which can be thought of as saying that $i^n_\infty$ is an approximate section
of the map $p^\infty_n$.

\begin{lemma}\label{lem:convergestoidentity}
  There is a constant $C$ such that for any $m \geq n_0$, any point $\xi \in
  \boundary G$, and point $x \in \Skel_2 K_n$ contained in a simplex
  of $K_n$ containing $p^\infty_n(\xi)$, the visual distance from $\xi$ to
  $i^\infty_n(x)$ is at most $Ca^{-m}$.
\end{lemma}

\begin{proof}
  Let $m$, $\xi$ and $x$ be as in the statement of the lemma.
  Let $\alpha$ be a geodesic ray with $\alpha(0) = e$ and $\alpha(\infty) =
  \xi$; then $\alpha(m)$ is a vertex of a simplex of $K_n$ containing
  $p^\infty_m(\xi)$. There is a path from $\alpha(m)$ to $x$ contained in the
  union of two 2-simplices of $K_n$, so by Lemma~\ref{cor:visualdiameter} it is
  sufficient to prove the lemma for $x = \alpha(m)$.

  Let $C$ be the constant appearing in Lemma~\ref{lem:linearproduct}.
  We have $\gromov{e}{\alpha(m)}{\alpha(n)} = m$ for $n\geq m$, and
  by~\cite[Proposition III.H.1.22]{bridsonhaefliger99} there is a constant
  $\delta'$ depending only on $\delta$ such that for $n\geq m$,
  \begin{align*}
    \mathrlap{\gromov{e}{\alpha(n)}{i^m_n(\alpha(m))}}\quad\quad &\\
          & \geq \min\{\gromov{e}{\alpha(n)}{\alpha(m)}, 
            \gromov{e}{\alpha(m)}{i^m_n(\alpha(m))}\} 
            - \delta' \\
          &\geq \min\{m, m - C\} - \delta' \\
          &=  m - C - \delta'.
  \end{align*}
  Therefore $\gromov{e}{\xi}{i^m_\infty(\alpha(m))} \geq m - C - \delta'$, 
  from which the claim follows.
\end{proof}

\subsection{Homotopies in $\boundary G$}

We now use the maps $i^n_\infty$ to construct homotopies in $\boundary G$. To
begin, we prove the following proposition, which constructs homotopies between
certain special maps.

\begin{proposition}\label{prop:basic_homotopies}
  Let $C$ be a positive number. Then there exist numbers $m_0$ and $\mu$
  with the following property. Let $n \geq m_0$ and let $\gamma_0$ and
  $\gamma_1$ be maps from $\boundary\Delta\to\boundary G$. Let $f_0$ and $f_1$
  be simplicial approximations to $p^\infty_n\composed \gamma_0$ and
  $p^\infty_n\composed\gamma_1$. Suppose that $\newdist(f_0(x), f_1(x))\leq C$
  for all $x \in \boundary \Delta$. Then there is a homotopy
  $h\from\boundary\Delta\times[0,1]\to\boundary G$ from $i^n_\infty\composed
  f_0$ to $i^n_\infty\composed f_1$ satisfying a bounded length condition: for
  any $x \in \boundary\Delta$,
  \begin{equation*}
    \Diam(h(\{x\} \times [0,1])) \leq \mu a^{-n},
  \end{equation*}
  where the diameter is taken with respect to the visual metric.
\end{proposition}

\begin{proof}
  Let $G$ satisfy $\ddag'_{C+D}$ with constants $L$ and $m_0 \geq n_0$.
  Let $M$ be such that $3\cdot2^M \geq 2^L+2$, so that $\sd^M\boundary\Delta$
  has at least $2^L+2$ edges.  By increasing $m_0$ we can also assume that $G$
  satisfies $\S_M$ with constants $L'$ and $m_0$.

  Let $n \geq m_0$ and let $f_0$ and $f_1$ be as in the statement of the
  proposition, so there exist $k_0$ and $k_1$ such that $f_0$ is simplicial on
  $\sd^{k_0}\boundary\Delta$ and $f_1$ is simplicial on
  $\sd^{k_1}\boundary\Delta$. We first show that we may assume that $k_0$ and
  $k_1$ are equal, at the cost of increasing the distance between $f_0$ and
  $f_1$ slightly. If they are not already equal, assume $k_0 < k_1$.

  Let $f_0'$ be the simplicial map on $\sd^{k_1}$ defined as follows. For each
  edge $I$ of $\sd^{k_0}\boundary\Delta$, choose an orientation of $I$ and
  define $f_0'\restricted{\sd^{k_1-k_0}I}$ to be given by $f_0\restricted{I}$
  on the initial edge of $\sd^{k_1-k_0}I$ and constant equal to the image under
  $f_0$ of the terminal vertex of $I$ on the rest of $\sd^{k_1-k_0}I$. Then
  there is a homotopy from $f_0$ to $f_0'$ such that for any $x \in
  \boundary\Delta$, the path traced by $x$ under the homotopy is contained in a
  single edge of $K_n$.  It follows that, when composed with $i^n_\infty$, this
  homotopy gives a homotopy from $i^n_\infty\composed f_0$ to $i^n_\infty
  \composed f_0'$ such that the path traced by any point in $\boundary\Delta$
  has diameter at most some multiple of $a^{-n}$ as given by
  Corollary~\ref{cor:visualdiameter}.  Notice also that $f_0'$ is still a
  simplicial approximation to $p^\infty_n\composed\gamma_0$.  Therefore we may
  replace $f_0$ by $f_0'$ and assume that $f_0$ and $f_1$ are simplicial maps
  on a common subdivision $J$ of $\boundary\Delta$, where now $\newdist(f_0(x),
  f_1(x))\leq C+D$ for any $x \in \boundary\Delta$.
  
  Then we have a map 
  \begin{equation*}
    f_0 \disjointunion f_1 \from J \times \{0\} \disjointunion J \times \{1\}
          \to K_n,
  \end{equation*}
  the domain of which is a simplicial complex that we think of as a subspace of
  the topological space $\boundary\Delta\times[0,1]$. This map is a
  simplicial approximation to the map $\gamma_0 \disjointunion \gamma_1$ from
  $J\times\{0\} \disjointunion J \times\{1\}$ to $\boundary G$. 
  
  Using $\ddag'_{C+D}$ we can extend $\gamma_0\disjointunion\gamma_1$ to a
  simplicial map
  \begin{equation*}
    J\times\{0\} \union(\Skel_0 J \times \sd^L [0,1]) \union J\times\{1\} \to K_n,
  \end{equation*}
  again thinking of the domain of this map as a subspace of
  $\boundary\Delta\times[0,1]$ in the same way. Each component of the
  complement of this subspace in $\boundary\Delta\times[0,1]$ is a $2$-cell
  whose boundary is an edge path of length $2^L+2$. Identifying this edge path
  with a copy of $\sd^M\boundary\Delta$ with some edges collapsed to a point,
  we can use the condition $\S_M$ to extend the map to the full $2$-simplex
  $\Delta$ in such a way that it is affine on $\sd^{L'}\Delta$, and use this
  extension to extend the map to all of $\boundary\Delta\times[0,1]$.
  
  It remains to prove the bounded length assertion.  Let $N$ be the number of
  $2$-simplices in $\sd^M\Delta$.  Corollary~\ref{cor:visualdiameter} gives a
  constant $C_0$ so that any $2$-simplex is mapped by $i^n_\infty$ to a set of
  diameter at most $C_0 a^{-n}$, and so we can take $\mu = NC_0$.
\end{proof}

The following proposition extends the family of maps between which we can build
homotopies.

\begin{proposition}\label{prop:uniform_homotopies}
  There exists numbers $m_0$ and $\nu$ with the following property. Let
  $\gamma\from\boundary\Delta\to\boundary G$ be a continuous map. Let $n \geq
  m_0$ and let $f$ a simplicial approximation to $p^\infty_n\composed\gamma$.
  Then there is a homotopy $h$ from $\gamma$ to $i^n_\infty\composed f$ such
  that for any $x \in \boundary\Delta$,
  \begin{equation*}
    \Diam(h(\{x\}\times [0,1])) \leq \nu a^{-n}.
  \end{equation*}
\end{proposition}

To prove this proposition, we use the following lemma.

\begin{lemma}\label{lem:name_this}
  There is a number $C$ such that, for any $\xi \in \boundary G$ and any $n
  \geq 0$ and any vertices $x_1$ and $x_2$ of the minimal simplices containing
  $i^n_{n+1}\composed p^\infty_n(\xi)$ and $p^\infty_{n+1}(\xi)$ respectively,
  \begin{equation*}
    \dist(x_1, x_2) \leq C.
  \end{equation*}
\end{lemma}

\begin{proof}
  By \cref{lem:close_projections} the distances $\newdist(p^\infty_n(\xi),
  p^\infty_{n+1}(\xi))$ are bounded by $6\delta + 3 + 2D$. Let $y$ be a vertex
  of a minimal simplex containing $p^\infty_n(\xi)$.
  
  Furthermore, $\dist(y, i^n_{n+1}(y))$ is bounded by $2\delta+1$, and
  $\newdist(i^n_{n+1}(y), i^n_{n+1}\composed p^\infty_n (\xi))$ is bounded by
  the product of $D$ and the the number of 2-simplices in $\sd^L\Delta$, where
  $L$ is the number appearing in Proposition~\ref{prop:i_in_dimension_2}.  By
  putting all this together we obtain a bound $C$ on
  $\newdist(i^n_{n+1}\composed p^\infty_n(\xi), p^\infty_{n+1}(\xi))$, and then
  $\dist(x_1, x_2) \leq C + 2D$.
\end{proof}

\begin{proof}[Proof of Proposition~\ref{prop:uniform_homotopies}]
  Apply Proposition~\ref{prop:basic_homotopies} with $C$ as in
  Lemma~\ref{lem:name_this} to obtain numbers $m_0$ and $\mu$.
  
  Let $\gamma\from\boundary\Delta\to\boundary G$, $n\geq m_0$ and
  $f\from\boundary\Delta\to K_n$ be as in the proposition. Let $f_0 = f$,
  and for $m > 0$ let $f_m\from\boundary\Delta\to K_{n+m}$ be a
  simplicial approximation to $p^\infty_{n+m}\composed\gamma$. For $m \geq 0$
  let $\gamma_m = i^{m+n}_\infty\composed f_m$. 

  For each $m$ we can apply Proposition~\ref{prop:basic_homotopies} to the maps
  $f_{m+1}$ and $i^{n+m}_{n+m+1}\composed f_m$. To see this, firstly
  Lemma~\ref{lem:name_this} gives the uniformly bounded distance assumption.
  Secondly, $f_{m+1}$ is a simplicial approximation to
  $p^\infty_{n+m+1}\composed\gamma$. Finally,
  $i^{n+m}_{n+m+1}\restricted{\Skel_1 K_{n+m}}$ is a simplicial approximation
  to a map $p^\infty_{n+m+1}\composed j_m$ for some map $j_m\from\Skel_1
  K_{n+m}\to \boundary G$, and then $i^{n+m}_{n+m+1}\composed f_m$ is a
  simplicial approximation to $p^\infty_{n+m+1}\composed j_m\composed f_m$.
  
  Therefore, for each $m$, \cref{prop:basic_homotopies} gives a homotopy $h_m$
  from $i^{n+m+1}_\infty\composed f_{m+1}$ to $i^{n+m}_\infty \composed f_m$
  satisfying a uniform continuity condition. Define an infinite concatenation
  of these homotopies:
  \begin{equation*}
    h(x,t) = 
    \begin{cases}
      \gamma(x) & \text{if } t = 0\\
      h_m\left(x,2^{m+1}\left(t-\frac{1}{2^{m+1}}\right)\right) 
              & \text{if $t \in \left[\frac{1}{2^{m+1}}, \frac{1}{2^m}\right]$
                    , $m\geq 0$}.
    \end{cases}
  \end{equation*}
  This function is clearly continuous except possibly on
  $\boundary\Delta\times\{0\}$. We now show that it is continuous there, too.

  Let $x \in \boundary\Delta$ and let $\epsilon > 0$. By
  Lemma~\ref{lem:convergestoidentity} there exists $M$ such that for $m \geq M$
  and $y \in \boundary \Delta$, $\rho(\gamma(y), \gamma_m(y)) \leq \epsilon/3$.
  Increasing $M$ if necessary we may assume that $\mu a^{-M-n} \leq
  \epsilon/3$.  Finally, let $U$ be a neighbourhood of $x$ such that
  $\rho(\gamma(x), \gamma(y)) \leq \epsilon/3$ for $y \in U$. Now, for any $y
  \in U$ and $t \in (0, 2^{-M}]$,
  \begin{equation*}
    \rho(h(x,0), h(y,t)) \leq \frac{\epsilon}{3} + \frac{\epsilon}{3} +
        \frac{\epsilon}{3} = \epsilon.
  \end{equation*}
  Therefore $h$ is continuous.

  To prove the final claim, we immediately see that for any $x \in
  \boundary\Delta$,
  \begin{equation*}
    \Diam(h(\{x\}\times (0,1])) 
            \leq \mu \left(\sum_{m\geq 0} a^{-m}\right)a^{-n},
  \end{equation*}
  and then $\Diam(h(\{x\}\times [0,1]))$ satisfies the same inequality by
  continuity of $h$.
\end{proof}

\Cref{thm:ddagimpliesLCtwo} follows from this proposition:

\ddagimpliesLCtwo*

\begin{proof}
  To begin, let $\nu$ and $m_0 \geq n_0$ be as in
  \cref{prop:uniform_homotopies}. Let $r_0 =
  k_1a^{-m_0}$, so that, by \cref{prop:projectstoasimplex},
  $p^\infty_{m_0}\composed\gamma$ has image contained in the star of a single
  vertex of $K_{m_0}$ for any map $\gamma\from\boundary\Delta\to\boundary
  G$ with $\Diam\image(\gamma)\leq r_0$.

  Let $\gamma\from\boundary\Delta\to\boundary G$ be such a map. Let 
  \begin{equation*}
    m = \floor{\log_a(k_1/\Diam\image\gamma)} \geq m_0.
  \end{equation*}
  Then a constant map $\boundary\Delta \to K_m$ is a simplicial approximation
  to $p^\infty_m\composed\gamma$; applying \cref{prop:uniform_homotopies} to
  this map immediately gives a homotopy $h$ from $\gamma$ to the constant map.

  Furthermore,
  \begin{align*}
    \Diam\image h &\leq \Diam\image\gamma + \sup_{x \in\boundary\Delta}
                                    \Diam(h(\{x\} \times [0,1])) \\
                  &\leq \Diam\image\gamma + a^{-m} \nu \\
                  &\leq (1 + a\nu/k_1) \Diam\image\gamma \qedhere
  \end{align*}
\end{proof}

\section{Conclusion}\label{sec:conclusion}

The results of Sections~\ref{sec:double_dagger}
and~\ref{sec:local_simple_connectedness} together imply one direction of
Theorem~\ref{thm:main_theorem}; we now complete the proof of this theorem.

\begin{proof}[Proof of Theorem~\ref{thm:main_theorem}]
  If $G$ is one-ended then $\ddag'_M$ is satisfied for any $M$ by
  Proposition~\ref{prop:sufficient_for_ddag_p}. Furthermore, if $\boundary G -
  \set\xi$ is simply connected for any $\xi \in \boundary G$ then the condition
  $\S_M$ is satisfied for any $M$ by
  Proposition~\ref{prop:sufficient_condition}. It follows that $\boundary G$ is
  locally simply connected by Theorem~\ref{thm:ddagimpliesLCtwo}.

  For the converse, suppose that $\boundary G$ is connected and locally simply
  connected. Let $\xi$ be a point in $\boundary G$ and let $\gamma\from
  S^1\to\boundary G - \set\xi$ be a continuous map. 

  The point $\xi$ is a conical limit point for the action of $G$ on $\boundary
  G$, so there exist points $\zeta_-$ and $\zeta_+$ in $\boundary G$ and a
  sequence $(g_n)_{n = 1}^\infty$ of elements of $G$ such that:
  \begin{enumerate}
    \item the sequence $g_n\cdot\xi$ converges to $\zeta_-$ as $n$ tends to
      infinity, and
    \item the sequence $(g_n)$, considered as continuous maps from $\boundary
      G$ to $\boundary G$, converges uniformly on compact subsets of $\boundary
      G - \set\xi$ to the constant map with image $\zeta_+$.
  \end{enumerate}
  See~\cite{bowditch99c} for further explanation of this property.

  Let $U$ be a neighbourhood of $\zeta_+$ such that the closure of $U$ does not
  contain $\zeta_-$ and let $V$ be a neighbourhood of $\zeta_+$ contained in
  $U$ such that any map from the circle to $V$ is null-homotopic in $U$. Then by
  the properties of the sequence $(g_n)$ described above, there exists an
  element $g_n$ of the sequence such that $g_n\cdot\xi \notin U$ and
  $g_n\composed\gamma$ has image in $V$. Then $g_n\composed\gamma$ admits a
  contracting homotopy with image in $U$, and the image of this homotopy under
  $g_n^{-1}$ is a contracting homotopy of $\gamma$ with image contained in
  $\boundary G - \set\xi$.
\end{proof}

If $G$ is a group such that $\boundary G - \set\xi$ is disconnected for some
point $\xi$ in $\boundary G$, arguments due to Bowditch~\cite{bowditch98b} and
Swarup~\cite{swarup96} build an action of $G$ on a tree-like space to show that
$\boundary G$ is disconnected. It follows that $\boundary G$ is locally path
connected if and only if either $G$ is one- or two-ended. No such argument is
available in the higher dimensional case, but on the other hand no known
example of a boundary of a hyperblic group contradicts the corresponding
assertion. Therefore the following question is natural.

\begin{question} 
  Is there a one-ended hyperbolic group $G$ such that $\boundary G$ is simply
  connected but not locally simply connected? 
\end{question}

Note that the converse question has a positive answer if virtually Fuchsian
groups (which have boundary homeomorphic to a circle) are excepted:

\begin{proposition}\label{prop:converse_theorem}
  If $G$ is a one-ended hyperbolic group such that $\boundary G$ is locally
  simply connected then $\boundary G$ is either simply connected or
  homeomorphic to a circle.
\end{proposition}

\begin{proof}
  By Theorem~\ref{thm:main_theorem}, all point complements in $\boundary G$ are
  simply connected. Let $p$ and $q$ be distinct points in $\boundary G$. If
  $\boundary G - \set{p,q}$ is connected then $\boundary G$ is simply connected
  by the Seifert--van Kampen theorem. It follows that either $\boundary G$ is
  simply connected, or the complement of any two points in $\boundary G$ is
  disconnected, in which case $\boundary G$ is a circle
  by a classical theorem in point set topology~\cite[II.2.13]{wilder49}.
\end{proof}

It is also natural to ask whether the theorem of this paper generalises to
higher dimensions. The difficulty is in generalising the proof of
Proposition~\ref{prop:sufficient_condition}: given a sequence of maps from $S^0$
to $\boundary G$ we pass to a uniformly convergent subsequence, but this is not
possible for higher dimensional spheres without controlling the modulus of
continuity of the maps.

\begin{question} Suppose that $\boundary G - \set\xi$ is $k$-connected for any
  point $\xi \in \boundary G$. Does it follow that $\boundary G$ is
  locally-$k$-connected?
\end{question}

\bibliography{references}
\end{document}